\documentclass[letterpaper, 10pt, conference]{ieeeconf}
\IEEEoverridecommandlockouts 
\usepackage{amsmath,amssymb,url}
\usepackage{graphicx,subfigure,color}

\newcommand{\parenth}[1]{\ensuremath{\left( #1 \right)}}

\newcommand{\refeqn}[1]{(\ref{eqn:#1})}

\newcommand{\SO}{\ensuremath{\mathsf{SO(3)}}}
\newcommand{\T}{\ensuremath{\mathsf{T}}}

\newcommand{\so}{\ensuremath{\mathfrak{so}(3)}}

\renewcommand{\Re}{\ensuremath{\mathbb{R}}}
\newcommand{\Sph}{\ensuremath{\mathsf{S}}}

\newcommand{\D}{\ensuremath{\mathbf{D}}}

\title{\LARGE \bf
Dynamics and Control of a Chain Pendulum on a Cart}

\author{Taeyoung Lee\authorrefmark{1}, Melvin Leok\authorrefmark{2}, and N. Harris McClamroch%
\thanks{Taeyoung Lee, Mechanical and Aerospace Engineering, George Washington University, Washington DC 20052 {\tt tylee@gwu.edu}}
\thanks{Melvin Leok, Mathematics, University of California at San Diego, La Jolla, CA 92093 {\tt mleok@math.ucsd.edu}}%
\thanks{N. Harris McClamroch, Aerospace Engineering, University of Michigan, Ann Arbor, MI 48109 {\tt
nhm@umich.edu}}%
\thanks{\textsuperscript{\footnotesize\ensuremath{*}}This research has been supported in part by NSF under grants CMMI-1243000 (transferred from 1029551).}
\thanks{\textsuperscript{\footnotesize\ensuremath{\dagger}}This research has been supported in part by NSF under grants DMS-0726263, DMS-1001521, DMS-1010687, and CMMI-1029445.}
}

\newtheorem{prop}{Proposition}

\graphicspath{{/Users/tylee/Documents/Research/Geometric\ Control/Simulations/Sphere/}}

\begin{document}
\allowdisplaybreaks
\maketitle \thispagestyle{empty} \pagestyle{empty}

\begin{abstract}
A geometric form of Euler-Lagrange equations is developed for a chain pendulum, a serial connection of $n$ rigid links connected by spherical joints, that is attached to a rigid cart.   The cart can translate in a horizontal plane acted on by a horizontal control force while the chain pendulum can undergo complex motion in 3D due to gravity.   The configuration of the system is in $(\Sph^2)^n \times \Re^2$.   We examine the rich structure of the uncontrolled system dynamics: the equilibria of the system correspond to any one of $2^n$ different chain pendulum configurations and any cart location. A linearization about each equilibrium, and the corresponding controllability criterion is provided. We also show that any equilibrium can be asymptotically stabilized by using a proportional-derivative type controller, and we provide a few numerical examples.
\end{abstract}

\section{Introduction}

Pendulum models have been a rich source of examples in nonlinear dynamics and control~\cite{FurP4ICDC03,MorNisIJC76}. For example, the dynamics of a double spherical pendulum and a Lagrange top have been studied in~\cite{BenSanDDNS06,MarSchZ93,CusMeeGSM90}. Generalized models, such as a 3D pendulum~\cite{ChaLeeJNS11} or a 3D pendulum attached to an elastic string~\cite{LeeLeoND11}, have been considered. A variety of control techniques have been applied, such as passivity-based approaches~\cite{ShiPogIJRNC00,SpoPIWC96}, swing-up strategies~\cite{ShiLudA04,AstFurA00}, Lyapunov-based method~\cite{GutAguAJC09}, controlled-Lagrangian~\cite{BloChaITAC01}, and hybrid control systems~\cite{ZhaSpoA01}. In particular, stabilization of a triple inverted pendulum has been studied in~\cite{HosKawA00,FurOchIJC84}.

In this paper, we consider the dynamics and control of a chain pendulum on a cart, that is a serial connection of $n$ rigid links, connected by a spherical joint, attached to a cart that moves on a horizontal plane. The configuration manifold of this system is $(\Sph^2)^n \times \Re^2$, where the manifold of unit vectors in $\Re^3$ is denoted by the two-sphere $\Sph^2$. 

Many interesting mechanical systems, such as robotic manipulators or variations of pendulum models, evolve on the two-sphere $\Sph^2$ or on products of two-spheres $(\Sph^2)^n$.  In most of the literature that treats dynamic systems on $(\Sph^2)^n$, including many of the above references, either $2n$ spherical polar angles or $n$ explicit equality constraints that enforce unit lengths are used to describe the configuration of the system. These descriptions necessarily involve complicated trigonometric expressions and introduce additional complexity in analysis and computations, as well as singularities.  

We demonstrate that globally valid Euler-Lagrange equations on $(\Sph^2)^n \times \Re^2$ can be developed for the chain pendulum on a cart system, and they can be analyzed in a compact form without local parameterization or constraints. This also leads a coordinate-free form of linearized equations, controllability criteria, and control systems. The main contribution of this paper is providing an intrinsic and unified framework to study dynamics and control of chain pendulum on a cart systems, that is uniformly applicable for an arbitrary number of links, and globally valid for any configuration of the links.



\section{Lagrangian Dynamics of the Chain Pendulum on a Cart System}\label{sec:EL}


\subsection{Background}



\setlength{\unitlength}{0.1\columnwidth}
\begin{figure}
\footnotesize\selectfont
\centerline{
\begin{picture}(4.3,5.2)(0,0)
\put(0,0){\includegraphics[width=0.43\columnwidth]{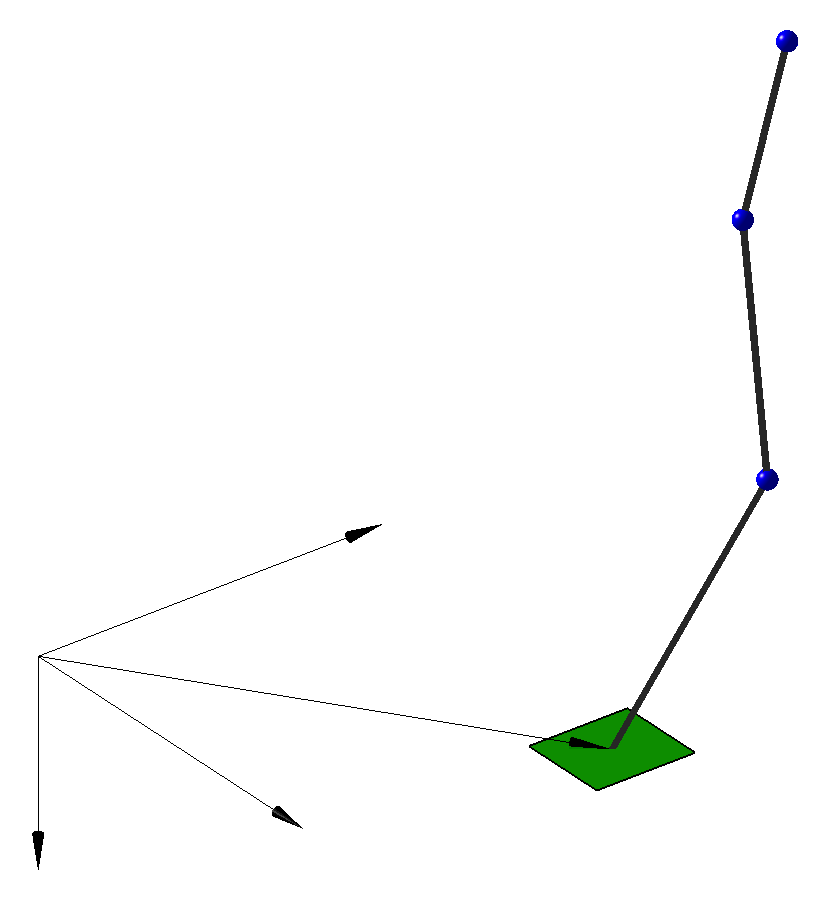}}
\put(1.8,2.1){$e_1$}
\put(1.4,0.15){$e_2$}
\put(-0.1,0.0){$e_3$}
\put(1.8,1.1){$x$}
\put(2.7,0.5){$m$}
\put(3.7,1.3){$l_1q_1$}
\put(4.1,2.13){$m_1$}
\put(3.3,2.8){$l_2q_2$}
\put(3.35,3.50){$m_2$}
\put(4.08,3.95){$l_3q_3$}
\put(4.2,4.5){$m_3$}
\put(3.2,0.82){\vector(2,-1){1.0}}
\put(4.25,0.2){$u$}
\end{picture}}
\caption{Chain pendulum on a cart ($n=3$)}
\end{figure}
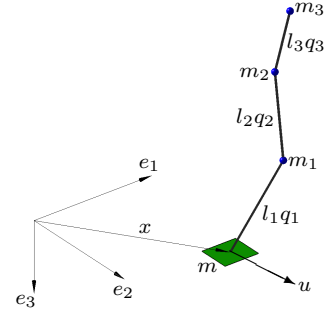

The cart of mass $m$ can translate on a horizontal plane, and its position in an inertial frame is denoted by $x\in\Re^2$. It is acted on by a horizontal control force $u\in\Re^2$.   A serial connection of $n$ rigid links, connected by spherical joints, is attached to the cart, where the mass of the $i$-th link is denoted by $m_i$ and the link length is denoted by $l_i$.  For simplicity, we assume that the mass of each link is concentrated at the outboard end of the link.    The direction vector of each link in an inertial frame is given by $q_i\in \Sph^2=\{q\in\Re^3\,|\, \|q\|=1\}$ for $i=1,\ldots,n$. The configuration manifold of this chain pendulum on a cart system is $(\Sph^2)^n\times\Re^2$.

The inertial frame is defined by the unit vectors $e_1=[1;\,0;\,0]\in\Re^3$, $e_2=[0;\,1;\,0]\in\Re^3$, $e_3=[0;\,0;\,1]\in\Re^3$; we assume that $e_3$ is in the direction of gravity.   Define $C=[e_1,\,e_2]\in\Re^{3\times 2}$. 

The kinematic equation for the direction vector of the $i$-th link $q_i$ is given by
\begin{align}
\dot q_i = \omega_i \times q_i=\hat\omega_i q_i,\label{eqn:dotqi}
\end{align}
where $\omega_i\in\Re^3$ is the angular velocity of the $i$-th link satisfying $\omega_i\cdot q_i=0$.

The \textit{hat} map $\wedge :\Re^{3}\rightarrow\so$ transforms a vector in $\Re^3$ to a $3\times 3$ skew-symmetric matrix, and is uniquely defined by the property that $\hat x y = x\times y$ for any $x,y\in\Re^3$. 
The inverse of the hat map is denoted by the \textit{vee} map $\vee:\so\rightarrow\Re^3$. 

Throughout this paper, the dot product of two vectors is denoted by $x\cdot y = x^T y$ for any $x,y\in\Re^n$. The $n\times n$ identity matrix is denoted by $I_n$. The $n\times m$ by matrix composed of zero elements is denoted by $0_{n\times m}$, and it is written as $0_n$ if $n=m$. A column-wise stack of matrices is written as $[A;B]=[A^T, B^T]^T\in\Re^{(a+b)\times n}$ for $A\in\Re^{a\times n}$ and $B\in\Re^{b\times n}$. Some properties of the hat map are given by
\begin{gather}
    \hat x y = x\times y = - y\times x = - \hat y x,\\
    x\cdot \hat y z = y\cdot \hat z x = z\cdot \hat x y\label{eqn:STP},\\
    \hat x\hat y z = (x\cdot z) y - (x\cdot y ) z,\label{eqn:VTP}\\
    C^TC=-C^T\hat e_3^2 C = I_2
\end{gather}
for any $x,y,z\in\Re^3$.

\subsection{Lagrangian}

The location of the cart is given by $Cx\in\Re^3$ in the inertial frame. Let $x_i\in\Re^3$ be the position of the outboard end of the $i$-th link in the inertial frame. It can be written as
\begin{align}
x_i &= Cx + \sum_{a=1}^i l_aq_a.\label{eqn:xi}
\end{align}
The total kinetic energy is composed of the kinetic energy of the cart and the kinetic energy of each mass:
\begin{align}
T 
& = \frac{1}{2} m\|\dot x\|^2 
+\frac{1}{2}\sum_{i=1}^n m_i \| C\dot x + \sum_{a=1}^i l_a\dot q_a\|^2.\label{eqn:T0}
\end{align}

For simplicity, we first consider the part of the kinetic energy, namely $T_x$ that is dependent on the motion of the cart. The part of \refeqn{T0} dependent on $\dot x$ is given by
\begin{align*}
T_x 
& = \frac{1}{2} (m + \sum_{i=1}^n m_i)\|\dot x\|^2 + C\dot x \cdot \sum_{i=1}^n\sum_{a=i}^n m_a l_i\dot q_i.
\end{align*}
This can be written as
\begin{align}
T_x = \frac{1}{2} M_{00}  \|\dot x\|^2 + \dot x \cdot \sum_{i=1}^n M_{0i} \dot q_i,\label{eqn:T1}
\end{align}
where the inertia matrices $M_{00}\in\Re$, $M_{0i}\in\Re^{2\times 3}$, and $M_{i0}\in\Re^{3\times 2}$ are given by
\begin{align}
M_{00} = m + \sum_{i=1}^n m_i,\quad M_{0i} = C^T \sum_{a=i}^n m_a l_i,\quad M_{i0} = M_{0i}^T\label{eqn:M0}
\end{align}
for $i=1,\ldots,n$. 
The part of the kinetic energy, namely $T_q$, that is independent of $\dot x$ is given by
\begin{align}
T_q & = \frac{1}{2}\sum_{i=1}^n m_i \| \sum_{a=1}^i l_a\dot q_a\|^2
 = \frac{1}{2} \sum_{i,j=1}^n M_{ij} \dot q_i \cdot \dot q_j,\label{eqn:T2}
\end{align}
where the inertia constants $M_{ij}\in\Re$ are given by
\begin{align}
M_{ij} =  \parenth{\sum_{a=\max\{i,j\}}^n m_a} l_il_j\label{eqn:Mij}
\end{align}
for $i,j=1,\ldots,n$. 

From \refeqn{T1} and \refeqn{T2}, the total kinetic energy is given by
\begin{align}
T=\frac{1}{2} M_{00}  \|\dot x\|^2 + \dot x \cdot \sum_{i=1}^n M_{0i}\dot q_i +\frac{1}{2} \sum_{i,j=1}^n M_{ij} \dot q_i \cdot \dot q_j,\label{eqn:T}
\end{align}
where the inertia matrices are given by \refeqn{M0}, \refeqn{Mij}. The gravitational potential energy is composed of the gravitational potential energy of each mass. From \refeqn{xi}, it can be written as
\begin{align}
V = -\sum_{i=1}^n m_i g x_i \cdot e_3 = - \sum_{i=1}^n \sum_{a=i}^n m_ag l_i e_3\cdot q_i.\label{eqn:V}
\end{align}
The Lagrangian of a chain pendulum on a cart is $L=T-V$.  

\subsection{Euler-Lagrange equations}


In~\cite{LeeLeoIJNME08}, it is shown that a coordinate-free form of Euler-Lagrange equations on the two-spheres can be derived from Hamilton's variational principle. The key idea is to express the variation of a curve on the two-sphere in terms of the exponential map on \SO. Let $q_i(t)$ be a curve on $\Sph^2$. Its variation can be written as
\begin{align*}
q^\epsilon_i(t) = \exp(\epsilon\xi_i(t)) q_i(t),
\end{align*}
where $\xi_i(t)$ is a curve in $\Re^3$ satisfying $q_i(t)\cdot \xi_i(t)=0$ for all $t$. As the exponential map represents the rotation of $q_i(t)$ about the axis $\xi(t)$ by the angle $|\epsilon|\|\xi_i(t)\|$ at each $t$, it is guaranteed that the varied curve $q^\epsilon_i(t)$ lies in $\Sph^2$. The corresponding infinitesimal variation is given by
\begin{align}
\delta q_i(t) = \frac{d}{d\epsilon}\bigg|_{\epsilon=0} \exp(\epsilon\xi_i(t)) q_i(t) = \xi_i(t)\times q_i(t),\label{eqn:delqi}
\end{align}
which lies in the tangent space $\T_{q_i(t)}\Sph^2$ as it is perpendicular to $q_i(t)$ at each $t$. Using these, we obtain the Euler-Lagrange equations of a chain pendulum on a cart as follows.

\begin{prop}\label{prop:EL}
Consider a chain pendulum on a cart, whose Lagrangian is given by \refeqn{T} and \refeqn{V}. The Euler-Lagrange equation on $(\Sph^2)^n\times\Re^2$ are as follows:
\begin{align}
&\begin{bmatrix}%
    M_{00} & M_{01} & M_{02} & \cdots & M_{0n} \\
    -\hat q_1^2 M_{10} & M_{11}I_{3} & -M_{12} \hat q_1^2 & \cdots & -M_{1n}\hat q_1^2\\%
    -\hat q_2^2 M_{20} & -M_{21} \hat q_2^2 & M_{22} I_{3} & \cdots & -M_{2n} \hat q_2^2\\%
    \vdots & \vdots & \vdots & & \vdots\\
    -\hat q_n^2 M_{n0} & -M_{n1} \hat q_n^2 & -M_{n2}\hat q_n^2 & \cdots & M_{nn} I_{3}
    \end{bmatrix}%
    \begin{bmatrix}
    \ddot x \\ \ddot q_1 \\ \ddot q_2 \\ \vdots \\ \ddot q_n
    \end{bmatrix}\nonumber\\
 &\qquad\quad=   \begin{bmatrix}
    u\\
    -\|\dot q_1\|^2M_{11} q_1 -\sum_{a=1}^n m_a gl_1\hat q_1^2 e_3\\
    -\|\dot q_2\|^2M_{22} q_2 -\sum_{a=2}^n m_a gl_2\hat q_2^2 e_3\\
    \vdots\\
    -\|\dot q_n\|^2M_{nn} q_n - m_n gl_n\hat q_n^2 e_3
    \end{bmatrix},\label{eqn:ELm}
\end{align}
where the inertia matrices are given by \refeqn{M0} and \refeqn{Mij}.

Or equivalently, it can be written in terms of the angular velocities as
\begin{gather}
\begin{bmatrix}%
    M_{00} & -M_{01}\hat q_1 & -M_{02}\hat q_2 & \cdots & -M_{0n}\hat q_n\\
    \hat q_1 M_{10} & M_{11}I_{3} & -M_{12} \hat q_1 \hat q_2 & \cdots & -M_{1n}\hat q_1 \hat q_n\\%
    \hat q_2 M_{20} &-M_{21} \hat q_2\hat q_1 & M_{22} I_{3} & \cdots & -M_{2n} \hat q_2 \hat q_n\\%
    \vdots & \vdots & \vdots & & \vdots\\
    \hat q_n M_{n0} &-M_{n1} \hat q_n \hat q_1 & -M_{n2}\hat q_n \hat q_2 & \cdots & M_{nn} I_{3}
    \end{bmatrix}%
    \begin{bmatrix}
    \ddot x \\ \dot \omega_1 \\ \dot \omega_2 \\ \vdots \\ \dot \omega_n
    \end{bmatrix}\nonumber\\
    =
    \begin{bmatrix}
    \sum_{j=1}^n M_{0j} \|\omega_j\|^2 q_j+u\\
    \sum_{j=2}^n M_{1j}\|\omega_j\|^2\hat q_1 q_j +\sum_{a=1}^n m_a gl_1\hat q_1 e_3\\
    \sum_{j=1,j\neq 2}^n M_{2j}\|\omega_j\|^2\hat q_2 q_j +\sum_{a=2}^n m_a gl_2\hat q_2 e_3\\
    \vdots\\
    \sum_{j=1}^{n-1} M_{nj}\|\omega_j\|^2\hat q_n q_j + m_n gl_n\hat q_n e_3\\
    \end{bmatrix},\label{eqn:ELwm}\\
\dot q_i = \omega_i\times q_i.\label{eqn:ELwm2}
\end{gather}
\end{prop}
\begin{proof}
See Appendix \ref{sec:pfEL}
\end{proof}

\subsection{Remarks about the Chain Pendulum on a Cart System}

The system of equations above for the chain pendulum on a cart system remains valid for arbitrary deformations of the chain with respect to the cart; part or all of the chain pendulum may be above the cart while part or all of the chain pendulum may be below the cart.   We ignore all possible collisions of the chain pendulum with itself and of the chain pendulum with the cart.    

The chain pendulum with the cart inertially fixed is a special case of the model above.   If there is a single link, $n=1$, the spherical pendulum is obtained; if there are two serial links, $n=2$, the double spherical pendulum is obtained.    The resulting dynamics of the chain pendulum for any number of links can be very complex.    The Euler-Lagrange equations for this case are easily obtained as special cases of the prior equations.   The cart is included in the system primarily as a means of achieving control actuation.

The planar chain pendulum on a cart is also a special case of the above model. Let $v\in\Re^3$ be a vector normal to $e_3$, i.e., $v\cdot e_3=0$. Suppose that the initial condition is chosen such that $\omega_i(0)\times v=0$ and $q_i(0)\cdot v=0$ for $i=1,\ldots,n$ and the control input is given such that $v\cdot Cu=0$. From \refeqn{ELwm}, we can show that the motion of links with respect to the cart remains confined to the plane normal to $v$.

In short, equations \refeqn{ELm}, or \refeqn{ELwm} and \refeqn{ELwm2}, provide the Euler-Lagrange equations for a chain pendulum or a chain pendulum on a cart with an arbitrary number of links. As they are developed in a compact, coordinate-free fashion, singularities that arise when using local coordinates are completely avoided.

\section{Dynamic Properties of the Uncontrolled Chain Pendulum on a Cart System}\label{sec:Lin}


\subsection{Equilibrium Configurations}

We can determine the equilibria of the uncontrolled chain pendulum on a cart system using \refeqn{ELwm} when $u=0$.   Assuming the configuration is constant, the equilibrium configurations in  $(\Sph^2)^n \times \Re^2$ are given by the cart in a fixed location and the chain pendulum aligned vertically, i.e., 
\begin{align*}
q_i \times e_3=0,\quad \text{for } i=1,\ldots,n.
\end{align*}

Consequently, assuming that we are interested in the equilibria where $x=0$, there are $2^n$ possible equilibrium configurations that correspond to the case where all $n$ links are vertical: $q_i=\pm e_3$ for $i=1,\ldots ,n$.   The equilibrium for which all links are aligned with the gravity direction, $q_i = e_3$ for $i=1,\ldots ,n$, is referred to as the \textit{hanging equilibrium}; the equilibrium for which all links are aligned opposite to the gravity direction $q_i = -e_3$ for $i=1,\ldots,n$, is referred to as the \textit{inverted equilibrium}.   Two other interesting equilibria correspond to the case that all adjacent links point in opposite directions, that is $q_i \cdot q_{i+1} = -1$.  These two equilibria, one with the first link pointing in the direction of $e_3$ and the other with the first link pointing in the direction of $-e_3$, are referred to as \textit{folded equilibria} since the chain pendulum is completely folded in each case.  

We introduce binary variables $s_i$ to describe a specific equilibrium. It is defined as
\begin{align}
s_i = 
\begin{cases}
1 &\text{if $q_i=e_3$ (along gravity),}\\
-1 &\text{if $q_i=-e_3$ (opposite to gravity)}
\end{cases}\label{eqn:si}
\end{align}
for $i=1,\ldots,n$. Then, an ordered $n$-tuple $s=(s_1,s_2,\ldots, s_n)$, specifies a particular equilibrium. 

\subsection{Linearized Equations of Motion}

The local dynamics near each equilibrium can be studied using linearization methods. From \refeqn{delqi}, the variation of $q_i$ from any equilibrium configuration can be written as
\begin{align*}
\delta q_i = \xi_i\times e_3\quad\text{or}\quad \delta q_i = \xi_i\times -e_3,
\end{align*}
where $\xi_i\in\Re^3$ with $\xi_i\cdot e_3=0$. The variation of $\omega_i$ is given by $\delta\omega\in\Re^3$ with $\delta\omega_i \cdot e_3=0$. Therefore, the third components of $\xi_i$ and $\delta\omega_i$ for any equilibrium configuration are zero, and they are omitted in the following linearized equation, i.e., the state vector of the linearized equation is composed of $C^T\xi_i\in\Re^2$. As a result, the dimension of the corresponding unconstrained linearized equation is $2n+2$.


\newcommand{\Mb}{\mathbf{M}}
\newcommand{\Kb}{\mathbf{K}}
\newcommand{\Bb}{\mathbf{B}}
\newcommand{\xb}{\mathbf{x}}
\newcommand{\ub}{\mathbf{u}}
\newcommand{\vb}{\mathbf{v}}
\newcommand{\Gb}{\mathbf{G}}

\begin{prop}\label{prop:Lin}
Consider an equilibrium of a chain pendulum on a cart, specified by $s=(s_1,\ldots,s_n)$ and $x=0$. The linearized equation of the Euler-Lagrange equation \refeqn{ELwm} is 
\begin{gather}
\Mb\ddot \xb  + \Gb\xb = \Bb u,\label{eqn:Lin}
\end{gather}
or equivalently 
\begin{align*}
\begin{bmatrix} \Mb_{xx} & \Mb_{xq}\\ \Mb_{qx} & \Mb_{qq} \end{bmatrix}
\begin{bmatrix}  \delta\ddot x \\ \ddot \xb_q\end{bmatrix}
+
\begin{bmatrix} 0_2 & 0_{2\times 2n}\\ 0_{2n\times 2} & \Gb_{qq}\end{bmatrix}
\begin{bmatrix}  \delta x \\ \xb_q\end{bmatrix}
=
\begin{bmatrix} I_2 \\ 0_{2n\times 2}\end{bmatrix}
u,
\end{align*}
where the corresponding sub-matrices are defined as
\begin{align*}
\xb_q & = [C^T \xi_1;\,\ldots\,;\,C^T \xi_n],\\
\Mb_{xx} &= M_{00}I_{2},\\
\Mb_{xq} &= \begin{bmatrix}
-s_1M_{01}\hat e_3C & -s_2M_{02}\hat e_3C & \cdots & -s_nM_{0n}\hat e_3C
\end{bmatrix},\\
\Mb_{qx} & = \Mb_{xq}^T,\\
\Mb_{qq} &=
    \begin{bmatrix}%
M_{11}I_{2} & s_{12}M_{12} I_2 & \cdots & s_{1n}M_{1n}I_2\\%
s_{21}M_{21} I_2 & M_{22} I_{2} & \cdots & s_{2n}M_{2n}I_2\\%
\vdots & \vdots & & \vdots\\
s_{n1}M_{n1}I_2 & s_{n2}M_{n2}I_2 & \cdots & M_{nn} I_{2}
    \end{bmatrix},\\
\Gb_{qq} & = \mathrm{diag}[s_1\sum_{a=1}^n m_a gl_1 I_2,\, \ldots, s_nm_ngl_nI_2],
\end{align*}
where $s_{ij}=s_is_j$ for $i,j=1,\ldots,n$.
\end{prop}
\begin{proof}
See Appendix \ref{sec:pfLin}.
\end{proof}

The local eigenstructure near each equilibrium can be determined from the linearized dynamics. The eigenvalues of \refeqn{Lin} are the roots of $\mathrm{det}[\lambda^2 \Mb + \Gb]=0$. Note that there are zero eigenvalues since the first two columns and rows of $\Gb$ are zero. These correspond to the cart dynamics.

For the hanging equilibrium, given by $s_i=1$ for all $i=1,\ldots,n$, the matrix $\Gb_{qq}$ becomes positive-definite. This implies that $\lambda^2 \leq 0$ always. Then, the stability of the nonlinear dynamics \refeqn{ELwm} is inconclusive from the linearized equation \refeqn{Lin} as $\mathrm{Re}[\lambda]=0$. But, it can be shown that the hanging equilibrium is stable in the sense of Lyapunov modulo the cart position, using the total energy as a Lyapunov function: the hanging equilibrium is the local minimum of the total energy, which is conserved along the solution of the uncontrolled system. For each of the other ${2^n}-1$ equilibrium configurations, there exist a pair of positive and negative eigenvalues, which implies that it is an unstable saddle. 
%
%
Heteroclinic connections between the various equilibria provide a non-local characterization of the dynamic flow.   

\section{Control Analysis}\label{sec:Ctrl}


Various control problems can be posed for the chain pendulum on a cart system.   For example, feedback control might be used to achieve asymptotic stabilization of any of the natural equilibrium solutions. From the linearized equation \refeqn{Lin}, we find a criterion for controllability as follows.
\begin{prop}\label{prop:Ctrl}
Consider an equilibrium of a chain pendulum on a cart, specified by $s=(s_1,\ldots,s_n)$ and $x=0$. Suppose that the inertia matrix $\Mb$ of the linearized equation \refeqn{Lin} is positive-definite. Then, the equilibrium is controllable if and only if the following subsystem is controllable
\begin{align}
\Mb_{qq} \ddot \xb_q + \Gb_{qq} = \Mb_{xq} \mathbf{u} \label{eqn:Ctrlxq}
\end{align}
for a control input $\ub\in\Re^2$, or equivalently 
\begin{align}
\mathrm{rank}[\lambda^2 \Mb_{qq} + \Gb_{qq},\; \Mb_{qx}]=2n\label{eqn:Ctrl}
\end{align}
for any $\lambda\in\mathbb{C}$. 
\end{prop}
\begin{proof}
See Appendix \ref{sec:pfCtrl}.
\end{proof}
This proposition states that the controllability of \refeqn{Lin} is equivalent to the controllability of a reduced system given by \refeqn{Ctrlxq}. It represents the dynamics of the chain pendulum, without the cart, where the input matrix is given by $\Mb_{qx}$ that corresponds to the inertia coupling between the chain dynamics and the cart dynamics. 

If the linearized equation about a specific equilibrium configuration is controllable, then we can design a control system to asymptotically stabilize that equilibrium.

\begin{prop}\label{prop:AS}
Consider an equilibrium of a chain pendulum on a cart, specified by $s=(s_1,\ldots,s_n)$ and $x=0$. Assume that the corresponding linearized equation is controllable (as characterized by Proposition \ref{prop:Ctrl}). The control force $u$ is chosen as follows:
\begin{align}
u = -K_x x - K_{\dot x} \dot x - \sum_{i=1}^n \{K_{q_i} C^T(s_ie_3\times q_i) + K_{\omega_i} C^T\omega_i\},\label{eqn:u}
\end{align}
for controller gains $K_x,K_{\dot x},K_{q_i},K_{\omega_i}\in\Re^{2\times 2}$ for $i=1,\ldots, n$. Then, there exist values of the controller gains such that the equilibrium of the controlled system is locally asymptotically stable. 
\end{prop}
\begin{proof}
See Appendix \ref{sec:pfAS}.
\end{proof}
This control system can be used uniformly to stabilize any of $2^n$ equilibrium configurations of a chain pendulum on a cart. But, as it is based on the linearized dynamics, the region of attraction could be limited. 

\section{Numerical Examples}\label{sec:NE}

In the subsequent numerical simulations, we consider a chain pendulum model with five identical links, i.e. $n=5$. This system has twelve degrees of freedom with two force control inputs on the cart. The properties of the cart and the links are chosen as
\begin{align*}
m=0.5\,\mathrm{kg},\quad m_i=0.1\,\mathrm{kg},\; l_i=0.1\,\mathrm{m}\;\text{ for $i=1,\ldots,5$}.
\end{align*} 
Throughout this section, the following units are used: $\mathrm{kg},\mathrm{m},\mathrm{sec}$ and $\mathrm{rad}$, unless otherwise specified. 

\begin{figure}
\centerline{
	\subfigure[Cart position $x$]{
		\includegraphics[width=0.5\columnwidth]{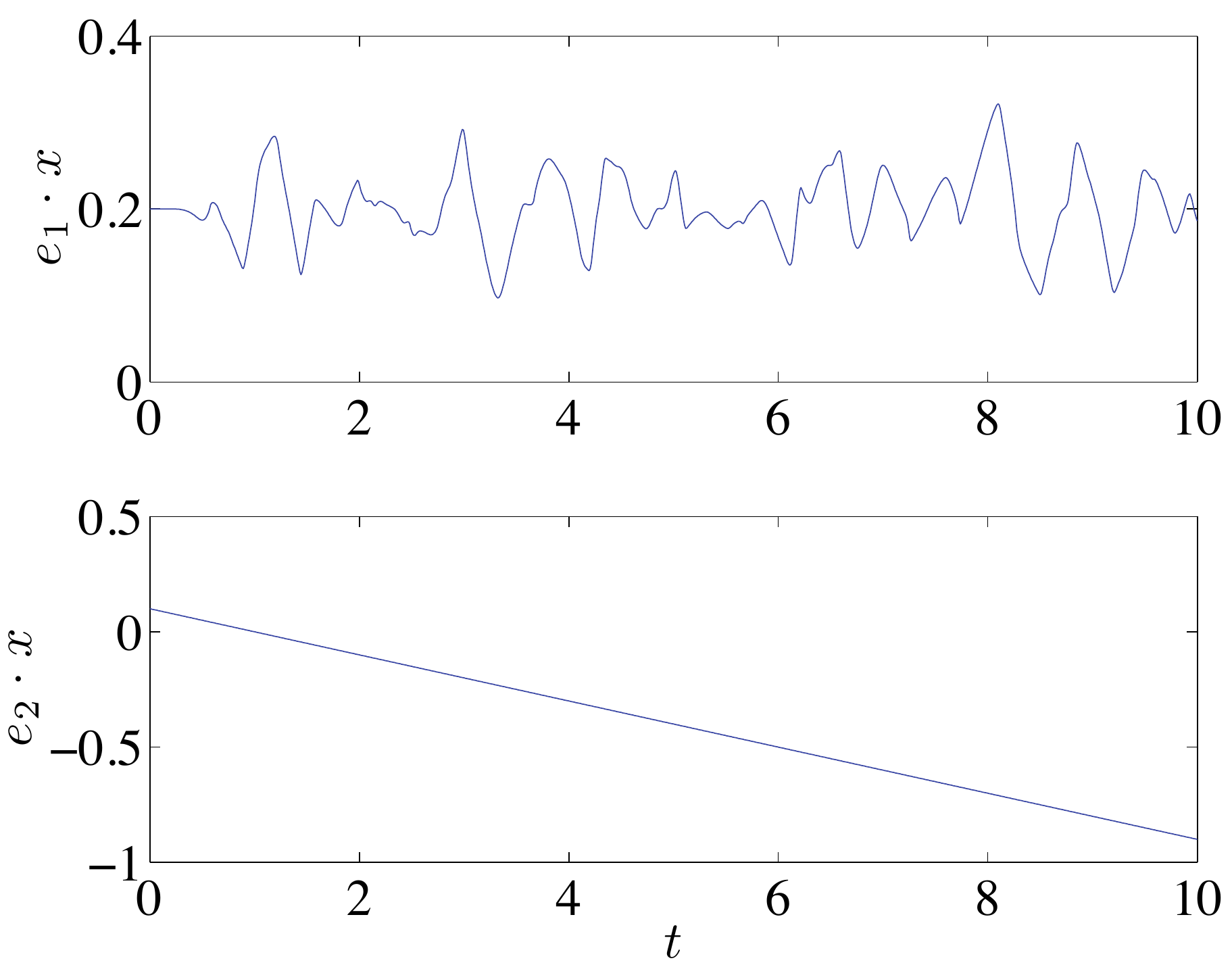}}
	\hfill
	\subfigure[Third elements of $q_3,q_4,q_5$]{
		\includegraphics[width=0.49\columnwidth]{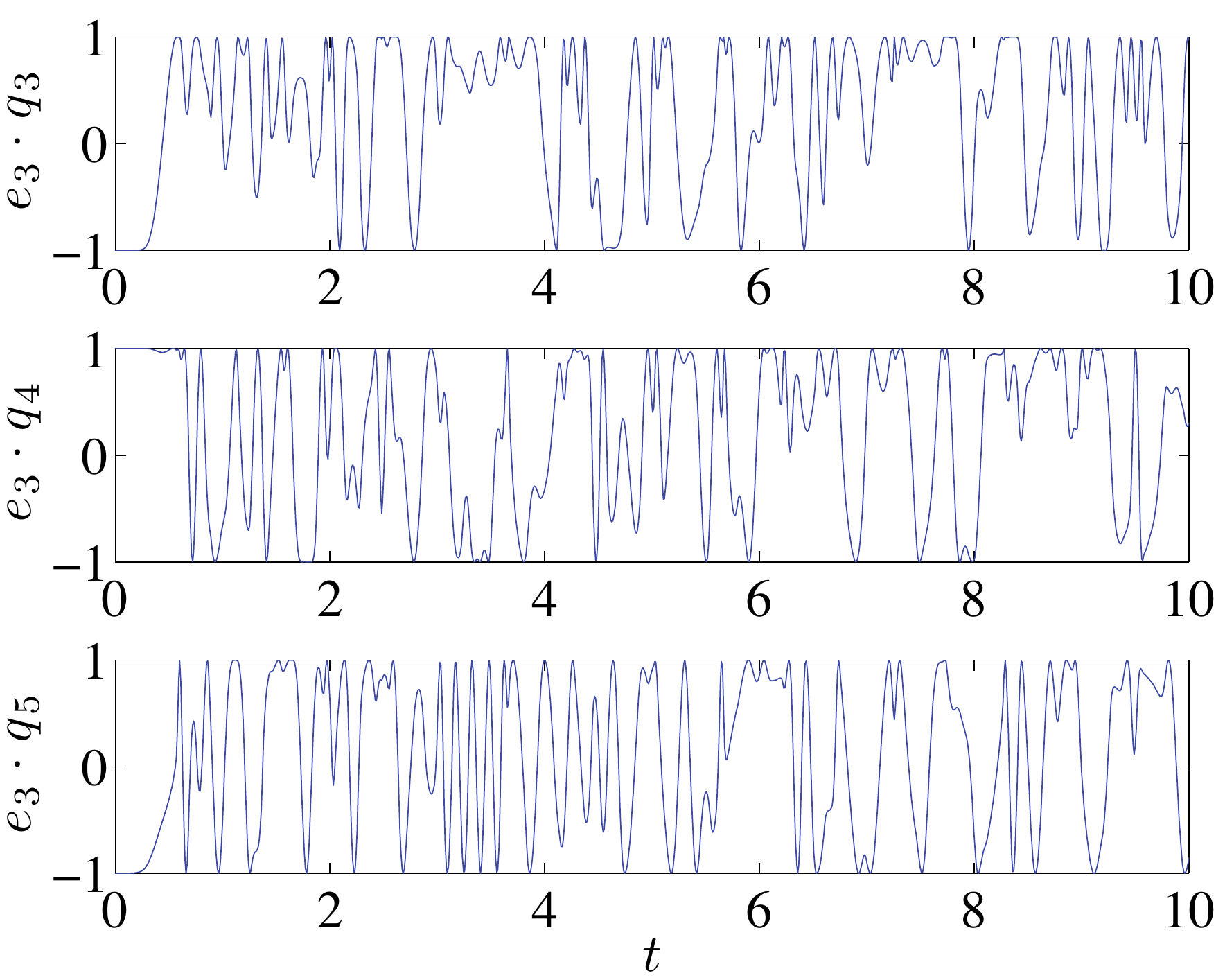}}
}
\centerline{
	\subfigure[Location of $m_5$ with respect to the cart ($\sum_{i=1}^5 l_i q_i$) in the $e_1e_3$ plane]{
		\includegraphics[width=0.49\columnwidth]{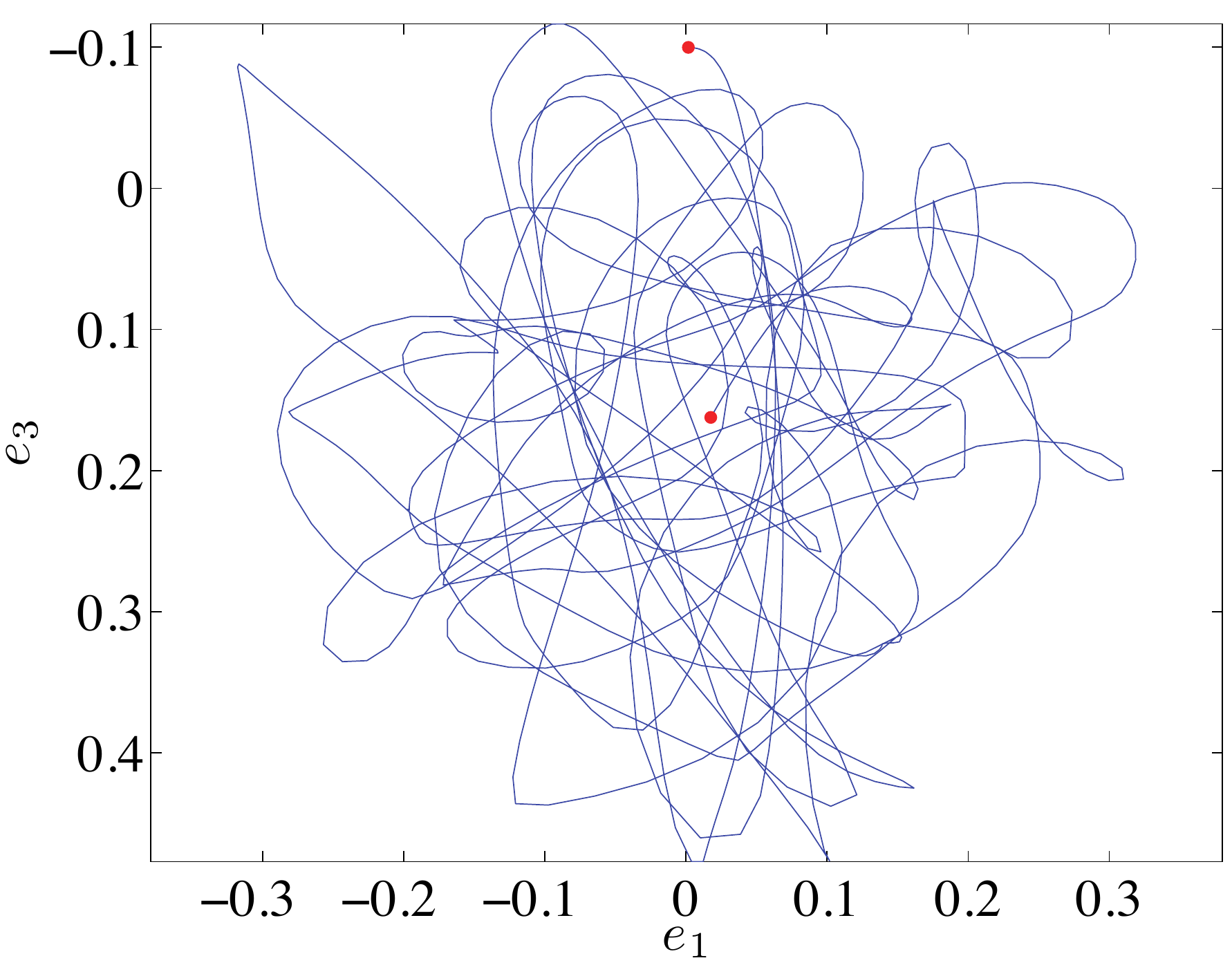}\label{fig:1c}}
	\hspace*{0.1cm}
	\subfigure[Kinetic energy $T$ ($T_q$ in \refeqn{T2}: red,dotted) and potential energy $V$]{
		\includegraphics[width=0.49\columnwidth]{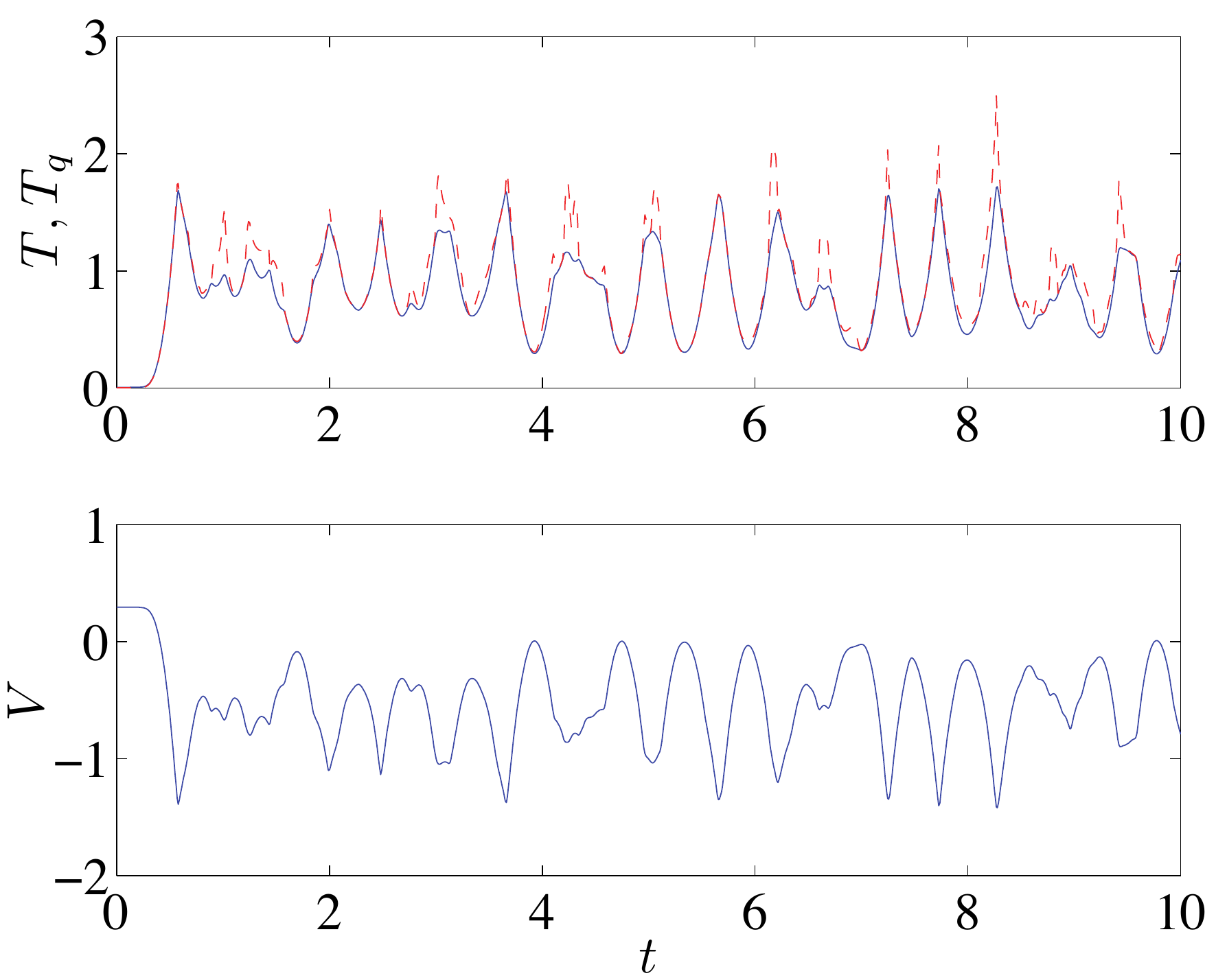}\label{fig:1d}}
}
\caption{Uncontrolled response: perturbation from a folded equilibrium}\label{fig:1}
\end{figure}

First, simulation results for the uncontrolled chain pendulum on a cart system are presented. The initial condition is a small perturbation from one of the completely folded equilibria, given by $s=(-1,1,-1,1,-1)$. More specifically,
\begin{gather}
\begin{gathered}
x(0)=[0.2;0.1],\quad \dot x(0)=[0;-0.1],\\
q_i(0)=s_ie_3 = (-1)^ie_3\text{ for $i=1,\ldots,4$,}\\
q_5(0)=[\sin 1^\circ;0;-\cos 1^\circ],\\
\omega_i(0)=0_{3\times 1}\text{ for $i=1,\ldots,5$},
\end{gathered}\label{eqn:IC1}
\end{gather}
where the fifth link is perturbed by $1^\circ$ from the equilibrium. The corresponding simulation results are shown at Figure \ref{fig:1}. The given initial condition guarantees that the relative motion of the links with respect to the cart always lies in the $e_1e_3$ plane, which is depicted in Figure \ref{fig:1c}. Figure \ref{fig:1d} demonstrates the complex transfer of potential energy and kinetic energy between the links and the cart as the chain pendulum on a cart dynamics evolve.

\begin{figure}
\centerline{
	\subfigure[Cart position $x$]{
		\includegraphics[width=0.5\columnwidth]{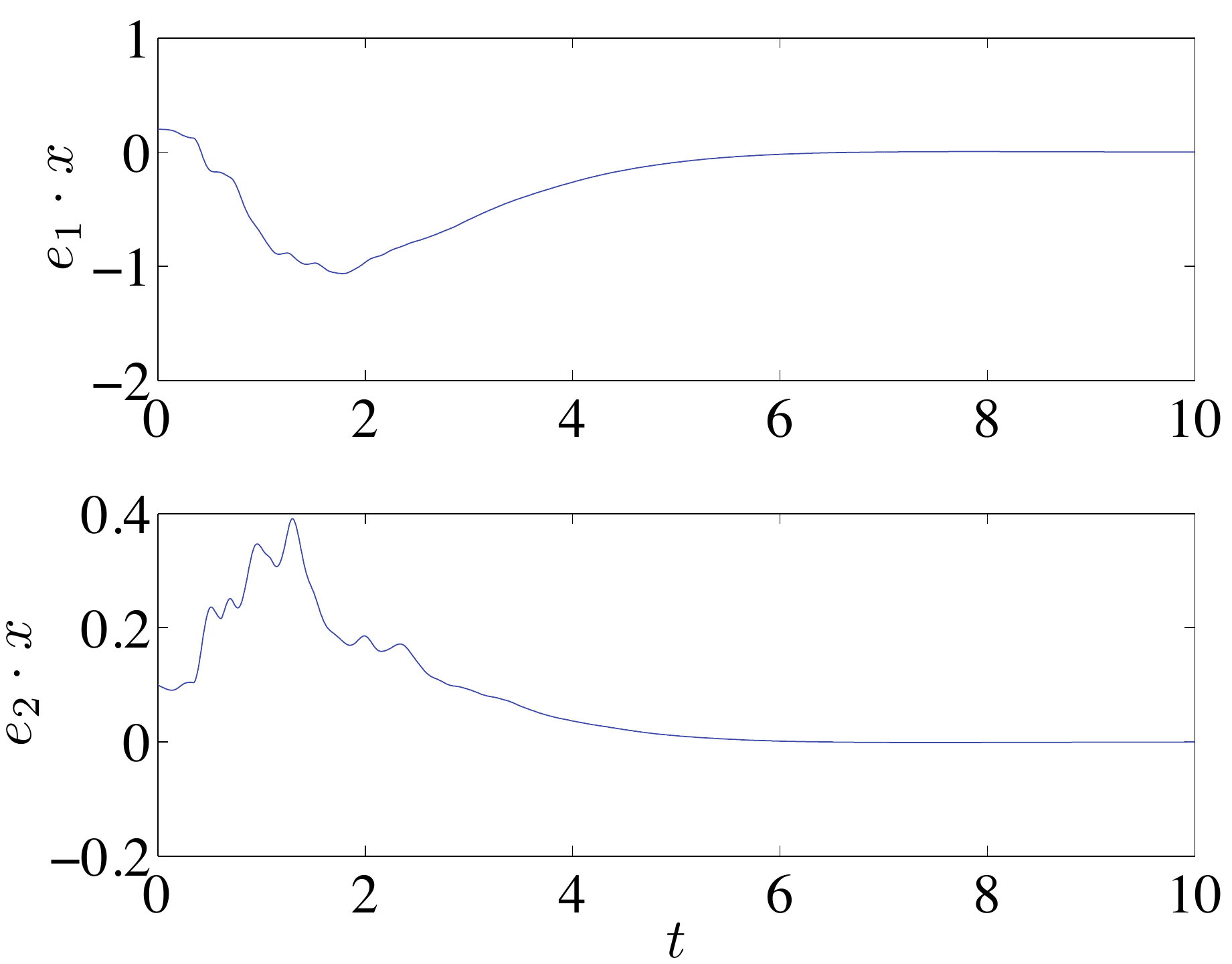}}
	\hfill
	\subfigure[Direction error $e_q$ and angular velocity error $e_\omega$ for links]{
		\includegraphics[width=0.495\columnwidth]{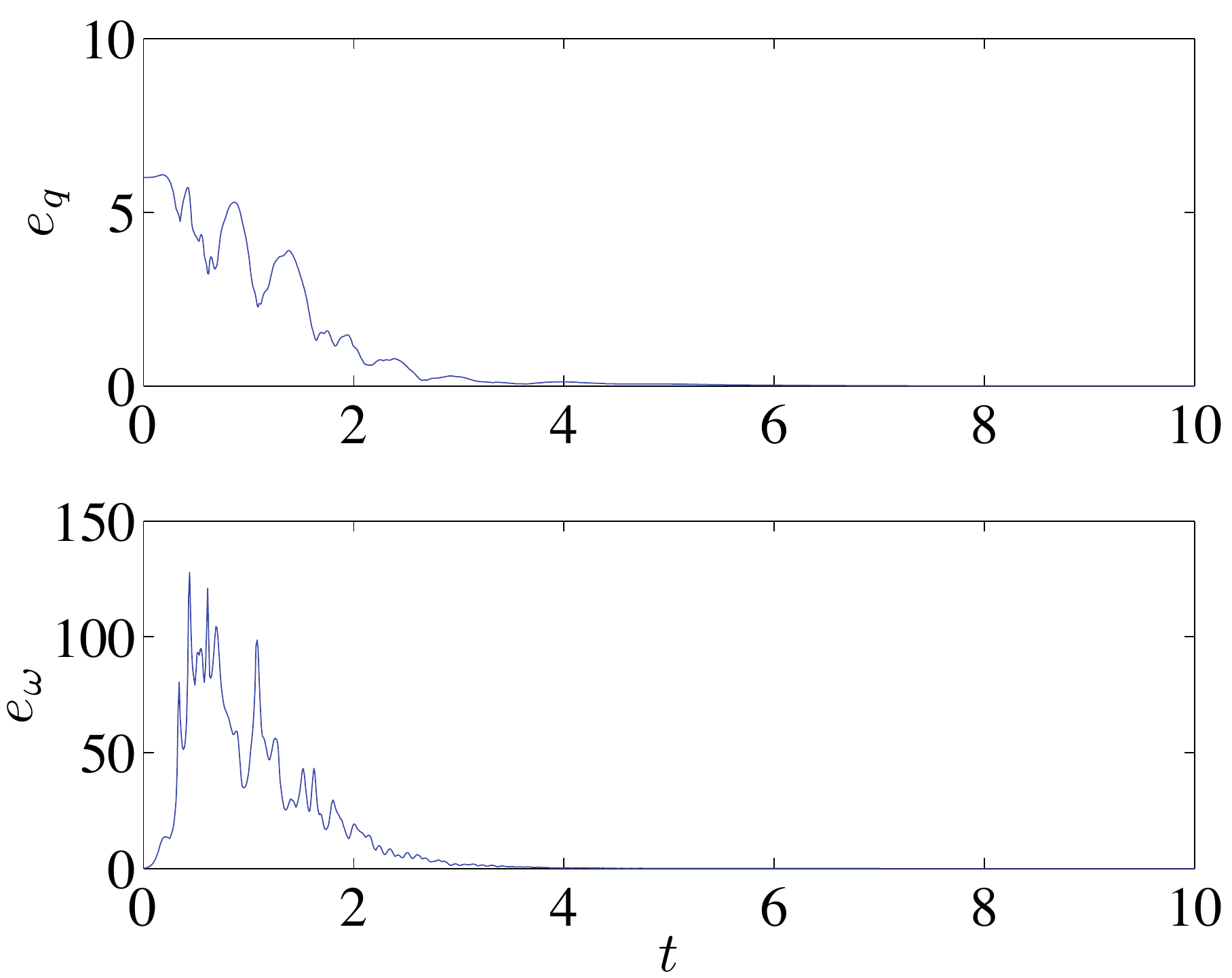}}
}
\centerline{
	\hspace*{0.2cm}
	\subfigure[Location of $m_5$ with respect to the cart ($\sum_{i=1}^5 l_i q_i$)]{
		\includegraphics[width=0.4\columnwidth]{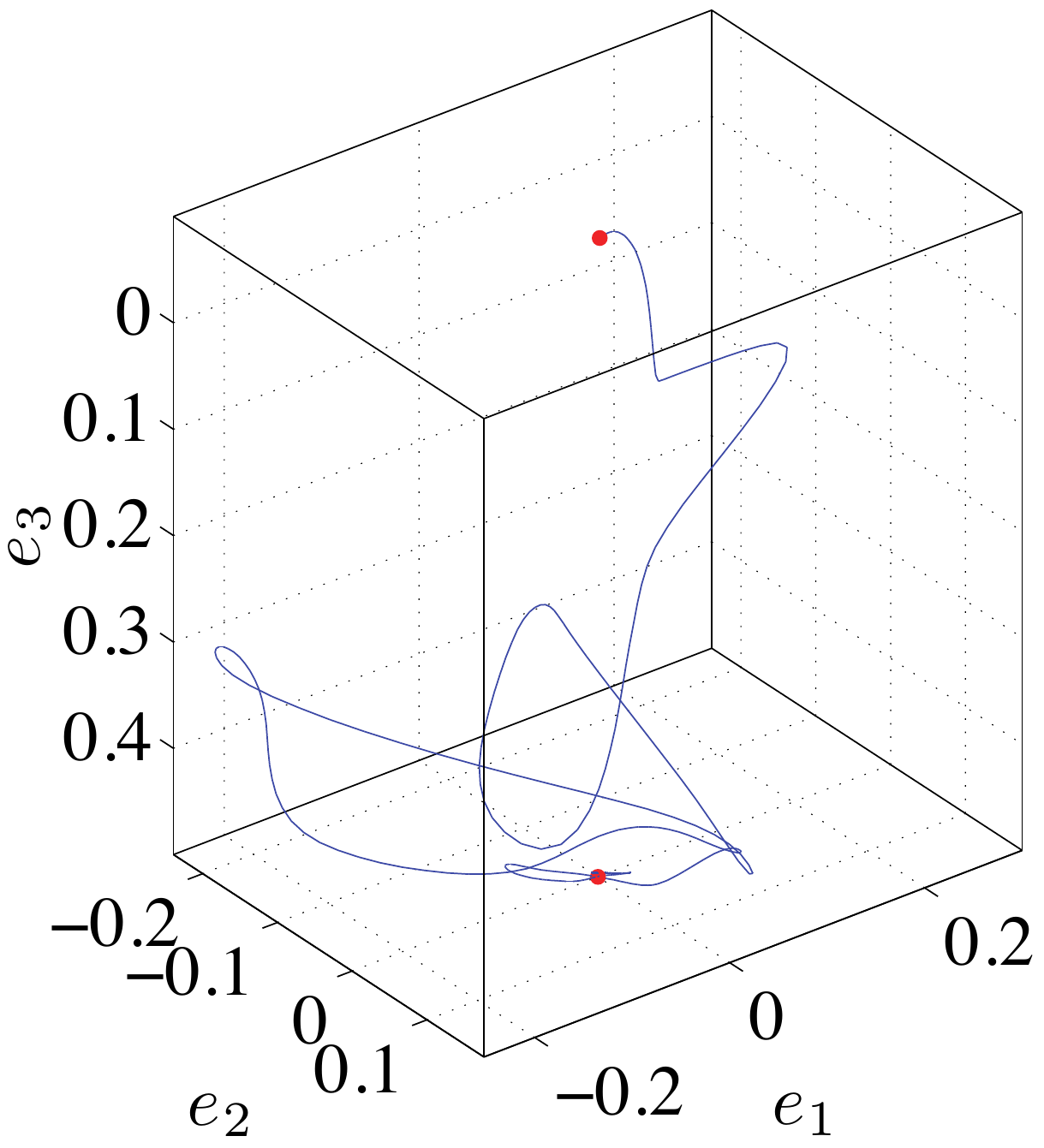}}
	\hspace*{0.4cm}
	\subfigure[Control force $u$]{
		\includegraphics[width=0.49\columnwidth]{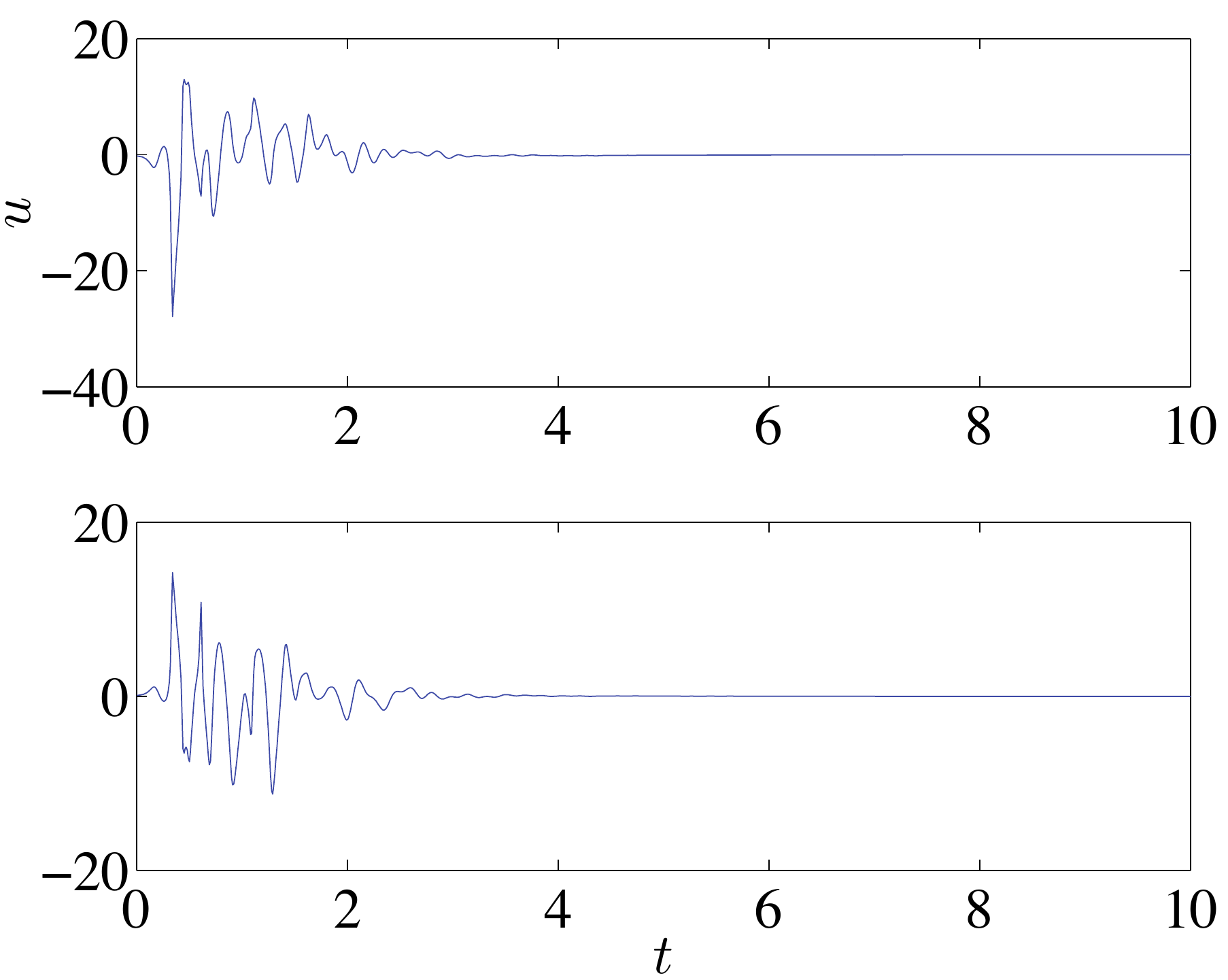}}
}
\caption{Controlled response: asymptotic stabilization of the hanging equilibrium $s=\{1,1,1,1,1\}$}\label{fig:2}
\end{figure}

\begin{figure}
\centerline{
	\subfigure[Cart position $x$]{
		\includegraphics[width=0.49\columnwidth]{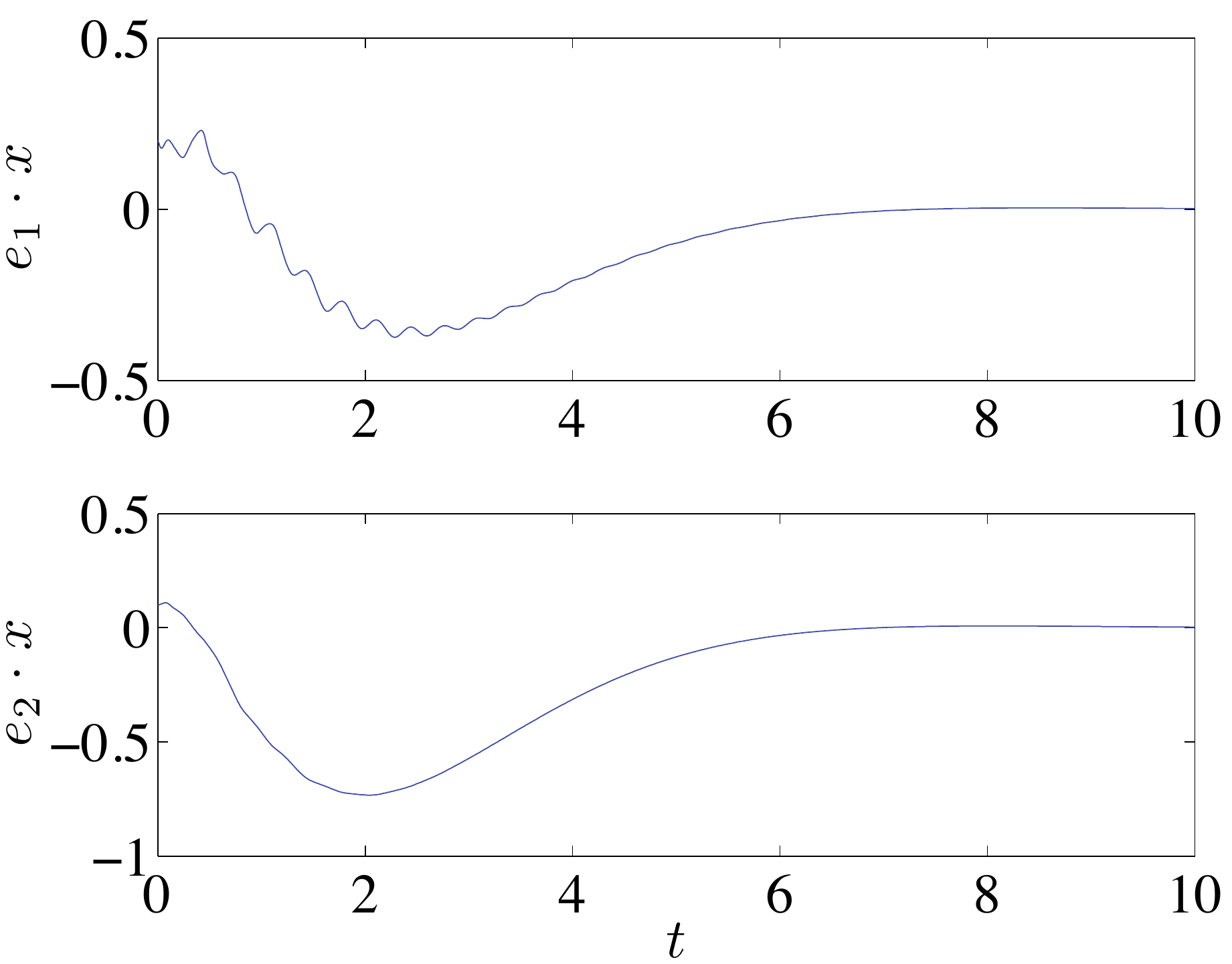}}
	\hfill
	\subfigure[Direction error $e_q$ and angular velocity error $e_\omega$ for links]{
		\includegraphics[width=0.475\columnwidth]{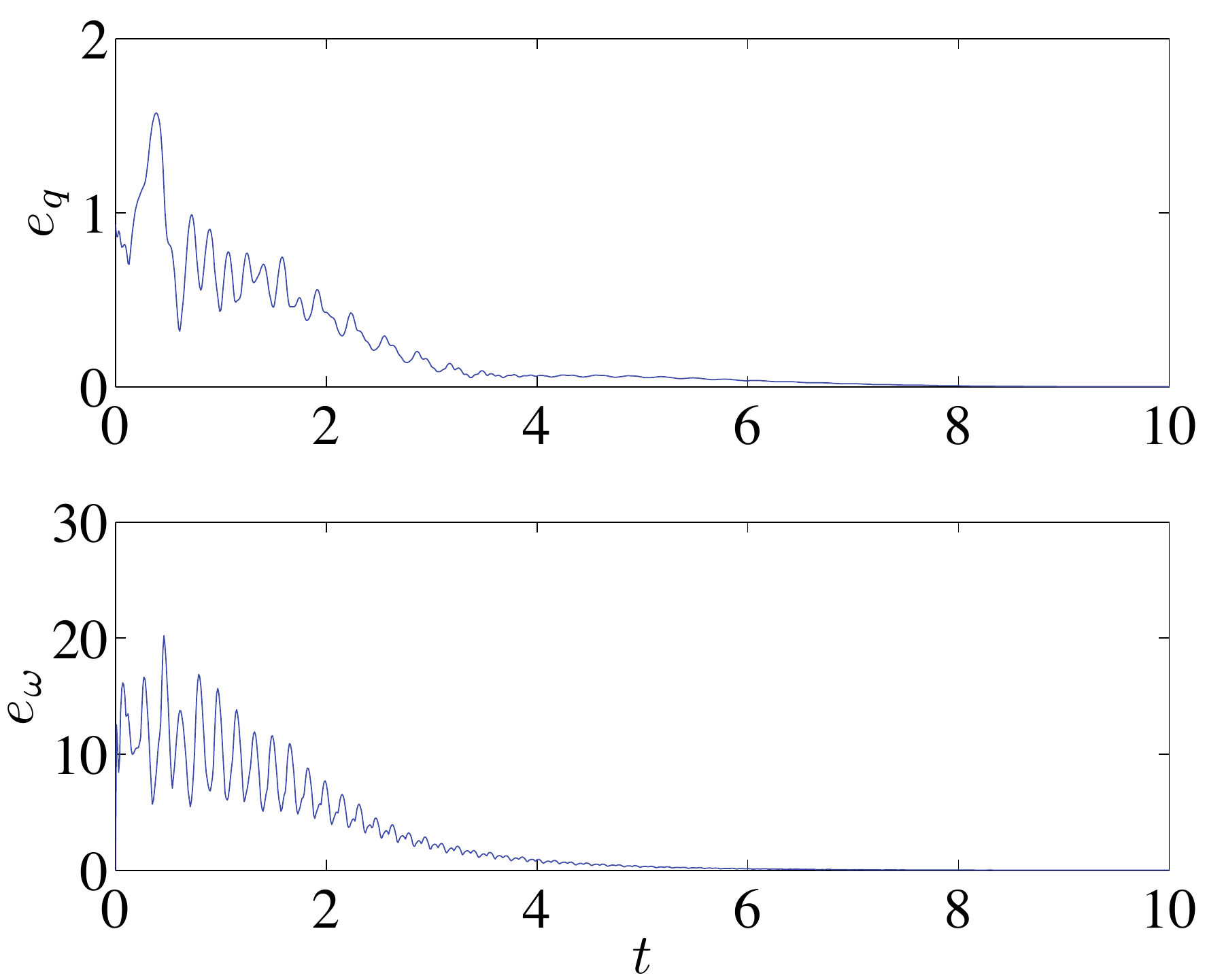}}
}
\centerline{
	\hspace*{0.1cm}
	\subfigure[Location of $m_5$ with respect to the cart ($\sum_{i=1}^5 l_i q_i$)]{
		\includegraphics[width=0.42\columnwidth]{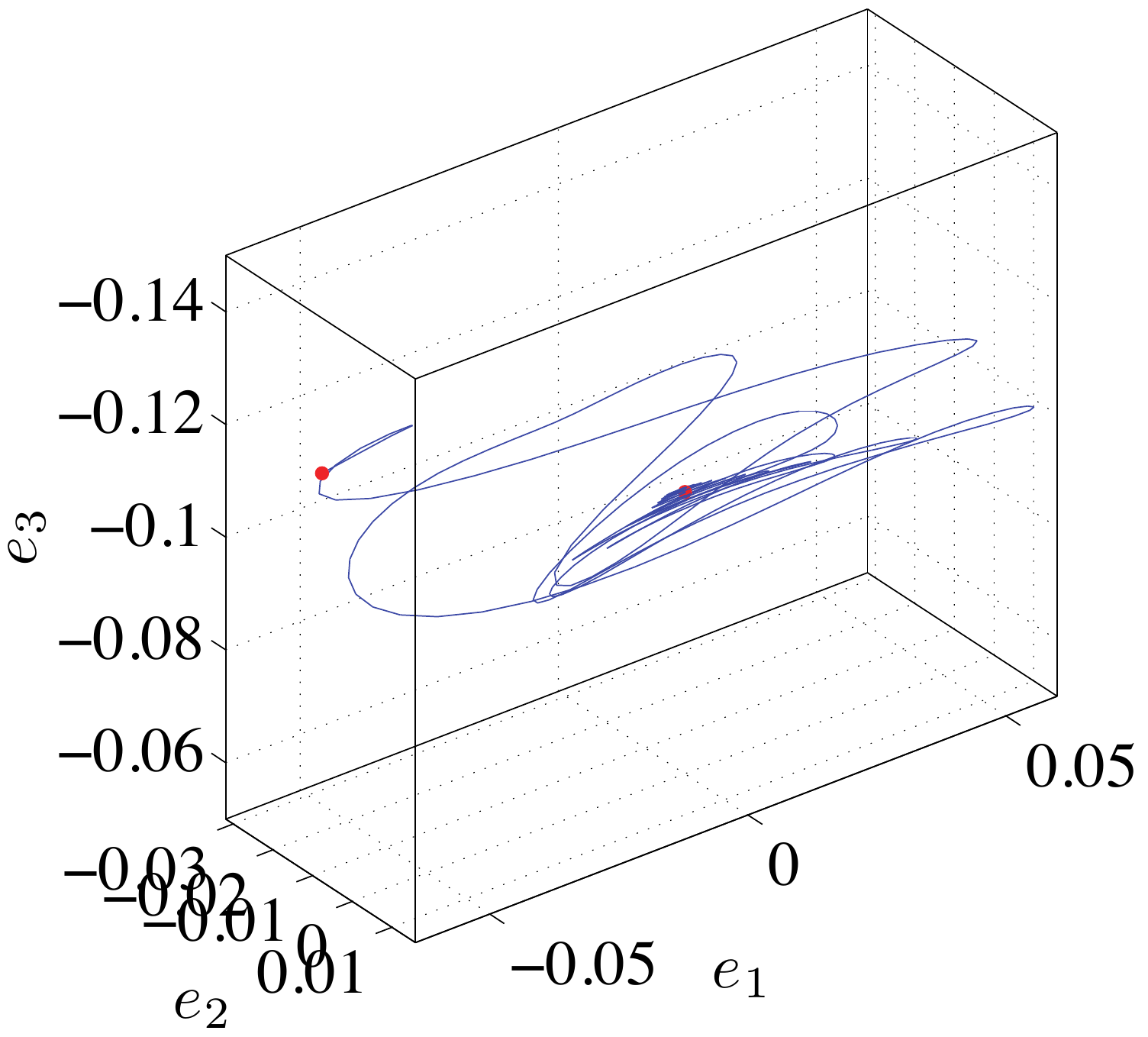}}
	\hspace*{0.2cm}
	\subfigure[Control force $u$]{
		\includegraphics[width=0.48\columnwidth]{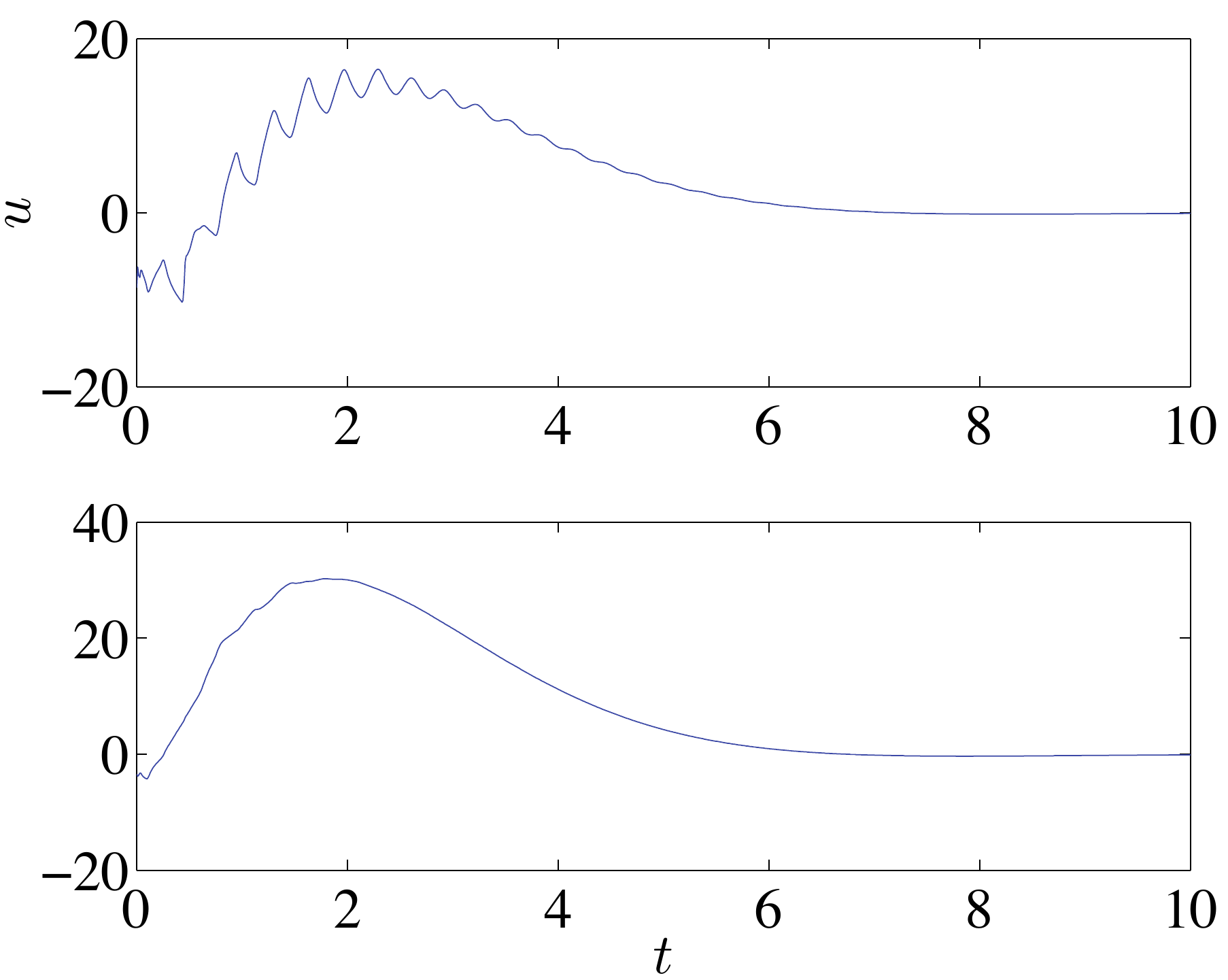}}
}
\caption{Controlled response: asymptotic stabilization of a partially-folded equilibrium $s=\{-1,-1,-1,1,1\}$}\label{fig:3}
\end{figure}

Second, simulation results are presented that show the response of the chain pendulum on a cart system to a feedback control \refeqn{u} that stabilizes the hanging equilibrium $s=(1,1,1,1,1)$. The initial conditions are same as \refeqn{IC1}, and the controller gains are chosen from a linear quadratic regulator with the weighting matrices $Q=\mathrm{diag}[8I_2,I_{2n},8I_2,I_{2n}]$ and $R=I_2$. We define the following variables that measure the direction errors and the angular velocity errors of links:
\begin{align*}
e_q = \sum_{i=1}^n \|q_i-s_ie_i\|,\quad e_\omega = \sum_{i=1}^n \|\omega_i\|.
\end{align*}
Figure \ref{fig:2} illustrates that the chain pendulum on a cart asymptotically approaches the hanging equilibrium. 

Next, we consider the control system \refeqn{u} to stabilize a partially folded equilibrium given by $s=(-1,-1,-1,1,1)$, i.e., the first three links are opposite to gravity, and the remaining last two links are aligned with gravity. Initial conditions are given as follows:
\begin{gather*}
q_1(0)=-[\sin6^\circ;0;\cos 6^\circ],\\
q_2(0)=q_3(0)=-[0;\sin4^\circ;0;\cos 4^\circ],\\
q_4(0)=[\sin5^\circ\cos4^\circ;-\sin5^\circ\sin4^\circ;\cos5^\circ],\\
q_5(0)=[-\sin35^\circ;0;\cos 35^\circ].
\end{gather*}
Other initial conditions for $x(0),\dot x(0)$ and $\omega_i(0)$ are identical to \refeqn{IC1}. Controller gains are chosen from a linear quadratic regulator with the weighting matrices $Q=\mathrm{diag}[I_2,8I_{2n},I_2,8I_{2n}]$ and $R=I_2$. Figure \ref{fig:3} illustrates that the cart and the pendulum asymptotically converge to the equilibrium $s=(-1,-1,-1,1,1)$.



\section{Conclusions}
Euler-Lagrange equations that evolve on $(\Sph^2)^n \times \Re^2$ have been derived for the chain pendulum on a cart system.   These equations of motion provide a remarkably compact form of the equations of motion, which enables us to analyze their dynamic properties and control systems uniformly for an arbitrary number of the links, and globally for any configuration of the links. 


We emphasize that modeling, analysis, and computations can be carried out directly in terms of a geometric coordinate-free framework as illustrated by the chain pendulum on a cart system studied in this paper; there is no need ever to use local angle coordinates.  This important fact is not appreciated by many researchers in dynamics and control who continue to formulate many dynamics and control problems in local coordinates, which could otherwise be analyzed with greater ease in a coordinate-free framework.

\appendices
\appendix
\renewcommand{\thesubsubsection}{\thesubsection.\arabic{subsubsection}}

\subsection{Proof of Proposition \ref{prop:EL}}\label{sec:pfEL}


The variations of the Lagrangian with respect to $x$ and $\dot x$ are given by
\begin{align*}
\D_x L \cdot \delta x& = 0,\\
\D_{\dot x} L \cdot \delta \dot x& = (M_{00} \dot x + \sum_{i=1}^n M_{0i}\dot q_i) \cdot \delta \dot x,
\end{align*}
where $\D_x L$ represents the derivative of $L$ with respect to $x$. From \refeqn{delqi}, the variation of $q_i$ is $\delta q_i =\xi_i\times q_i$ for $\xi_i\in\Re^3$ with $\xi_i\cdot q_i=0$.
The variation of the Lagrangian with respect to $q_i$ is given by
\begin{align*}
\D_{q_i}L\cdot \delta q_i & =  \sum_{a=i}^n m_a gl_i e_3 \cdot (\xi_i\times q_i) = -\sum_{a=i}^n m_a gl_i\hat e_3 q_i \cdot \xi_i,
\end{align*}
where \refeqn{STP} has been used. The variation of $\dot q_i$ is given by
\begin{align*}
\delta \dot q_i = \dot \xi_i \times q_i + \xi_\times \dot q_i.
\end{align*}
From this and \refeqn{STP}, the variation of the Lagrangian with respect to $\dot q_i$ is given by
\begin{align*}
&\D_{\dot q_i}L\cdot \delta \dot q_i  
 = (M_{i0}\dot x + \sum_{j=1}^n M_{ij}\dot q_j) \cdot (\dot \xi_i \times q + \xi_i \times \dot q_i)\\
& = \hat q_i (M_{i0}\dot x + \sum_{j=1}^n M_{ij}\dot q_j)\cdot \dot\xi_i +
\hat{\dot q}_i (M_{i0}\dot x + \sum_{j=1}^n M_{ij}\dot q_j)\cdot \xi_i.
\end{align*}


Using these expressions and integrating by parts, the variation of the action integral can be written as
\begin{align*}
\delta \mathfrak{G} &= \int_{t_0}^{t_f} -\{M_{00} \ddot x + \sum_{i=1}^n M_{0i}\ddot q_i\}\cdot \delta x \\
&+ \sum_{i=1}^n\{-\hat q_i (M_{i0}\ddot x + \sum_{j=1}^n M_{ij}\ddot q_j)-\sum_{a=i}^n m_a gl_i\hat e_3 q_i \}\cdot \xi_i \,dt.
\end{align*}

According to the Lagrange-d'Alembert principle, the sum of the variation of the action integral, and the integral of the virtual work done by the control force on the cart, namely $\int_{t_0}^{t_f} {u \delta x}\,dt$, is zero. This implies that the expression within the first pair of braces in the above equation is equal to $-u$, and the expression within the second pair of braces is parallel to $q_i$ for any $1\leq i\leq n$, as $\xi_i$ is perpendicular to $q_i$. Therefore, we obtain
\begin{gather}
M_{00} \ddot x + \sum_{i=1}^n M_{0i}\ddot q_i =u,\label{eqn:xddot0}\\
-\hat q_i^2 (M_{i0}\ddot x + \sum_{j=1}^n M_{ij}\ddot q_j)+\sum_{a=i}^n m_a gl_i\hat q_i^2 e_3=0.\label{eqn:qddot0}
\end{gather}
Equation \refeqn{qddot0} is rewritten to obtain an explicit expression for $\ddot q_i$. As $q_i\cdot \dot q_i =0$, we have $\dot q_i \cdot \dot q_i +q_i\cdot \ddot q_i=0$. Using this and \refeqn{VTP}, we have
\begin{align*}
-\hat q_i^2 \ddot q_i = -(q_i\cdot \ddot q_i )q_i + (q_i\cdot q_i)\ddot q_i =(\dot q_i \cdot \dot q_i) q_i + \ddot q_i.
\end{align*}
Substituting this into \refeqn{qddot0},
\begin{align}
M_{ii} & \ddot q_i  -\hat q_i^2 (M_{i0}\ddot x + \sum_{\substack{j=1\\j\neq i}}^n M_{ij}\ddot q_j)\nonumber\\
&=- M_{ii}\|\dot q_i\|^2 q_i-\sum_{a=i}^n m_a gl_i\hat q_i^2 e_3.\label{eqn:qddot1}
\end{align}
These equations \refeqn{xddot0} and \refeqn{qddot1} are rewritten in a matrix form to obtain \refeqn{ELm}.

These can also be rewritten in terms of the angular velocities. Since $\dot q_i = \omega_i\times q_i$ for the angular velocity $\omega_i$ satisfying $q_i\cdot\omega_i=0$, we have
\begin{align*}
    \ddot q_i & = \dot \omega_i \times q_i + \omega_i\times (\omega_i\times q_i) 
    =-\hat q_i \dot\omega_i - \|\omega_i\|^2q_i.
\end{align*}
Substituting this into \refeqn{xddot0} and \refeqn{qddot1}, and using the fact that $\dot\omega_i\cdot q_i=0$, we obtain \refeqn{ELwm}.

\subsection{Proof of Proposition \ref{prop:Lin}}\label{sec:pfLin}

Consider the hanging equilibrium where $s=(1,1,\ldots,1)$, and $x=0$. The variations from the hanging equilibrium are
\begin{align*}
x^\epsilon = \epsilon\delta x,\quad \dot x^\epsilon = \epsilon\delta\dot x,\quad
q_i^\epsilon = \exp(\epsilon \hat \xi_i)e_3,\quad \omega_i^\epsilon = \epsilon\delta\omega_i,
\end{align*}
where $\delta x,\delta\dot x\in\Re^2$, and $\xi_i,\delta\omega_i\in\Re^3$ with $\xi_i\cdot e_3=0$ and $\delta\omega_i\cdot e_3=0$. This yields the following infinitesimal variation $\delta q_i = \xi_i\times e_3$. From \refeqn{dotqi}, $\delta\dot q_i$ is given by
\begin{align*}
\delta \dot q_i = \dot\xi_i \times e_3 =\delta\omega_i \times e_3 + 0\times (\xi_i\times e_3)=\delta\omega_i \times e_3.
\end{align*}
Since both sides of the above equation is perpendicular to $e_3$, this is equivalent to $e_3\times(\dot\xi_i\times e_3) = e_3\times(\delta\omega_i\times e_3)$, which yields
\begin{gather*}
\dot \xi - (e_3\cdot\dot\xi) e_3 = \delta\omega_i -(e_3\cdot\delta\omega_i)e_3.
\end{gather*}
Since $\xi_i\cdot e_3 =0$, we have $\dot\xi\cdot e_3=0$. As $e_3\cdot\delta\omega_i=0$ from the constraint, we obtain the linearized equation for \refeqn{dotqi}:
\begin{align}
\dot\xi_i = \delta\omega_i.\label{eqn:dotxii}
\end{align}

Substituting these into \refeqn{ELwm}, and ignoring the higher order terms, we obtain
\begin{align}
&    \begin{bmatrix}%
    M_{00}I_{2\times 2} & -M_{01}\hat e_3C & -M_{02}\hat e_3C & \cdots & -M_{0n}\hat e_3C\\
    C^T\hat e_3 M_{10} & M_{11}I_{2} & M_{12} I_2 & \cdots & M_{1n}I_2\\%
    C^T\hat e_3 M_{20} &M_{21} I_2 & M_{22} I_{2} & \cdots & M_{2n}I_2\\%
    \vdots & \vdots & \vdots & & \vdots\\
    C^T\hat e_3 M_{n0} &M_{n1}I_2 & M_{n2}I_2 & \cdots & M_{nn} I_{2}
    \end{bmatrix}\nonumber\\
&\qquad\qquad \times     
    \begin{bmatrix}
    \delta\ddot x \\ C^T\ddot \xi_1 \\ C^T\ddot \xi_2 \\ \vdots \\ C^T\ddot \xi_n
    \end{bmatrix}
   =
    \begin{bmatrix}
    u\\
    -\sum_{a=1}^n m_a gl_1\xi'_1\\
    -\sum_{a=2}^n m_a gl_2\xi'_2\\
    \vdots\\
    -m_n gl_n\hat \xi'_n\\
    \end{bmatrix},\label{eqn:Lin0}
\end{align}
where we have used the fact that $\hat e_3^2 = \mathrm{diag}[-1,-1,0]$, $C^T\hat e_3^2 C=-I_2$ and $\hat e_3 C C^T = \hat e_3$. This is the linearized equation about the hanging equilibrium. 

This analysis can be easily generalized to other equilibria, where one or more links are aligned opposite to the gravity. Consider the equilibrium where only the $i$-th link is pointing upward, i.e. $q_i=e_3$ and $q_j=-e_3$ for all $j\neq i$. By following the same procedure, we obtain the same form of the linearized equation as \refeqn{Lin0}, where all of the terms related to $M_{ij}$, $M_{ji}$ and $l_i$ for all $j\neq i$ are multiplied by $-1$. This yields \refeqn{Lin}.

\subsection{Proof of Proposition \ref{prop:Ctrl}}\label{sec:pfCtrl}

Suppose that $\Mb$ is invertible. It is well-known that the linearized system \refeqn{Lin} is controllable, if and only if 
\begin{align}
\mathrm{rank}[\lambda^2 \Mb + \Gb,\; \mathbf{B}]=2n+2\label{eqn:Ctrl0}
\end{align}
for any generalized eigenvalue $\lambda$ satisfying $\mathrm{det}[\lambda^2 \Mb + \Gb]=0$~(see \cite{HugSkeAJAM80,LauArnITAC84}). This is a generalization of the Popov-Belevitch-Hautus (PBH) eigenvalue test to a second-order system. While it is not explicitly stated in the above references~\cite{HugSkeAJAM80,LauArnITAC84}, it is straightforward to find an equivalent condition in terms of eigenvectors, which is similar to the PBH eigenvector test. 

We claim that \textit{\refeqn{Ctrl0} holds if and only if there is no generalized left eigenvector that is orthogonal to $\Bb$, i.e. for any non-zero eigenvector $\vb_i\in\Re^{2n+2}$ satisfying $\vb_i^T(\lambda_i^2 \Mb  + \Gb) =0$, we have $\vb_i^T \Bb \neq 0_{1\times 2}$.} The proof is as follows:\\
\noindent(\textit{Sufficiency}) Suppose that there is a generalized eigenvector $\vb_i$ that is orthogonal to $\Bb$. Left-multiplying \refeqn{Ctrl0} by $\vb_i$ yields
\begin{align*}
\vb_i^T [\lambda^2 \Mb + \Gb,\; \mathbf{B}]= [ \vb_i^T(\lambda^2 \Mb  + \Gb),\, \vb_i^T\Bb],
\end{align*}
which becomes $[0_{1\times 2n+2},\,0_{1\times 2}]$ when $\lambda=\lambda_i$. Therefore, the matrix given in \refeqn{Ctrl0} has linearly dependent rows, which implies that it is rank-deficient.\\
\noindent(\textit{Necessity}) If the matrix given in \refeqn{Ctrl0} is rank-deficient for $\lambda_i$, there exists a vector $\vb_i$ satisfying 
\begin{align*}
\vb_i^T [\lambda_i^2 \Mb + \Gb,\; \mathbf{B}]=[ ((\lambda^2 \Mb  + \Gb) \vb_i)^T,\, \vb_i^T\Bb]=[0_{1\times 2n+4}],
\end{align*}
which implies that $\vb_i$ is a generalized eigenvector that is orthogonal to $\Bb$.

Using this eigenvector test, we show that \refeqn{Ctrl} implies \refeqn{Ctrl0}. More specifically, we show that if \refeqn{Ctrl0} is false, then \refeqn{Ctrl} is false. Suppose that there exists a generalized eigenvector $\vb_i=[\vb_x;\vb_q]$ that is orthogonal to $\Bb$. Then, we have $\vb_i^T \Bb = \vb_x^T I_2 + \vb_q^T 0_{2n\times 2}=\vb_x^T=0_{1\times 2}$ from the definition of $\Bb$ in \refeqn{Lin}. As $\vb_i$ is the left eigenvector, we also have
\begin{align}
\vb_i^T & [\lambda_i^2\Mb+\Gb]  =[0_{2\times 1}^T,\, \vb_q^T] \begin{bmatrix} \lambda_i^2\Mb_x & \lambda_i^2\Mb_{xq}\\ \lambda_i^2\Mb_{qx} & \lambda_i^2\Mb_{qq}+\Gb_{qq}\end{bmatrix}\nonumber\\
& =
\begin{bmatrix} \lambda_i^2\vb_q^T\Mb_{qx}
& \vb_q^T(\lambda_i^2\Mb_{qq}+\Gb_{qq})\end{bmatrix}\nonumber\\
& = \begin{bmatrix} 0_{1\times 2} & 0_{1\times 2n}\end{bmatrix}.\label{eqn:Ctrl2}
\end{align}
When $\lambda_i=0$, this yields $\vb_q^T\Gb_{qq}=0_{1\times 2n}\Rightarrow \vb_q=0_{2n\times 1}$ as $\Gb_{qq}$ is invertible. This is not possible since it contradicts the fact that $\vb_i=[0_{2\times 1};\vb_q]\neq 0_{2n+2\times 1}$. Therefore, $\lambda_i\neq 0$. Then, \refeqn{Ctrl2} implies
\begin{align}
\vb_q^T(\lambda_i^2\Mb_{qq}+\Gb_{qq}) = 0_{1\times 2n},\quad
\vb_q^T\Mb_{qx} = 0_{1\times 2},\label{eqn:Ctrl3}
\end{align}
which states that there exists a non-zero generalized left eigenvector of \refeqn{Ctrlxq} that is orthogonal to $\Mb_{qx}$. Therefore \refeqn{Ctrl} is false.

As a last step, we show \refeqn{Ctrl0} implies \refeqn{Ctrl}. If \refeqn{Ctrl} is false, there exists a non-zero eigenvector $\vb_q$ satisfying \refeqn{Ctrl3}. Let $\vb_i=[0_{2\times 1};\vb_q]$. Then, it is orthogonal to $\Bb$ as $\vb_i^T \Bb=0_{1\times 2}$. And $\vb_i$ is a generalized left eigenvector of $(\Mb,\Gb)$ as it satisfies \refeqn{Ctrl2}. In short, \refeqn{Ctrl} is equivalent to \refeqn{Ctrl0}.

\subsection{Proof of Proposition \ref{prop:AS}}\label{sec:pfAS}

\renewcommand{\Kb}{\mathbf{K}}


Using \refeqn{delqi}, \refeqn{dotxii}, the linearized control input is given by
\begin{align*}
u & = -K_x \delta x - K_{\dot x} \delta \dot x  - \sum_{i=1}^n \{K_{q_i} C^T\xi_i + K_{\omega_i} C^T\delta\omega_i\}.
\end{align*}
From the definition of the state vector $\xb=[\delta x; C^T\xi_1;\ldots;C^T\xi_n]$ in \refeqn{Lin}, $u$ can be written as
\begin{align*}
u = -\Kb_\xb \xb - \Kb_{\dot \xb} \dot\xb, 
\end{align*}
where $\Kb_\xb=[K_x,K_{q_i},\ldots, K_{q_n}]\in\Re^{2\times 2n+2}$ and $\Kb_{\dot \xb}=[K_{\dot x},K_{\omega_1},\ldots, K_{\omega_n}]\in\Re^{2\times 2n+2}$. Since \refeqn{Lin} is controllable, we can choose the controller gains $\Kb_{\xb},\Kb_{\dot \xb}$ such that the equilibrium is asymptotically stable for the linearized equation \refeqn{Lin}. According to Theorem 4.7 in \cite{Kha02}, the equilibrium becomes asymptotically stable for the nonlinear Euler-Lagrange equation \refeqn{ELwm}.

\vfill
\bibliography{CDC12.1}
\bibliographystyle{IEEEtran}

\end{document}